\documentclass[12pt]{amsproc}
  \usepackage[margin=1in]{geometry}
\usepackage{setspace,fullpage}
\usepackage{graphicx}
\usepackage[nice]{nicefrac}
\usepackage{amssymb}
\usepackage{amsmath,amsthm,amscd,amssymb, mathrsfs}

\usepackage{latexsym}
\usepackage[colorlinks,citecolor=red,pagebackref,hypertexnames=false]{hyperref}

\numberwithin{equation}{section}

\theoremstyle{plain}
\newtheorem{theorem}{Theorem}[section]
\newtheorem{lemma}[theorem]{Lemma}
\newtheorem{corollary}[theorem]{Corollary}

\theoremstyle{definition}

\theoremstyle{remark}
\newtheorem{remark}[theorem]{Remark}

\newtheorem{case[theorem]}{Case}

\title[\parbox{14cm}{\centering{ Extension and averaging operators for finite fields \hspace{1in}}} \quad]{Extension and averaging operators for finite fields}
\author{ Doowon Koh and Chun-Yen Shen}

\address{Department of Mathematics\\
Chungbuk National University \\
Cheongju city, Chungbuk-Do 361-763  Korea}
\email{koh131@chungbuk.ac.kr}

\address{Department of Mathematics and Statistics\\
 McMaster University\\
Hamilton, L8S 4K1 Canada}
\email{shenc@umail.iu.edu}



\begin{document}

\begin{abstract} In this paper we study  $L^p-L^r$ estimates of both the extension operator and the averaging operator associated with the algebraic variety $S=\{x\in {\mathbb F}_q^d: Q(x)=0\}$ where $Q(x)$ is a non-degenerate quadratic form over the finite field ${\mathbb F}_q$ with $q$ elements. We show that the Fourier decay estimate on $S$ is good enough to establish the sharp averaging estimates in odd dimensions.
In addition, the Fourier decay estimate enables us to simply extend the sharp $L^2-L^4$ conical extension result in $\mathbb F_q^3$, due to Mockenhaupt and Tao, to the  $L^2-L^{2(d+1)/(d-1)}$ estimate in all odd dimensions $d\geq 3.$  We also establish  sharp estimates except for endpoints  of the mapping properties of the average operators in the case that when the variety $S$  in even dimensions $d\geq 4$  contains a $d/2$-dimensional subspace.

\end{abstract}
     
 \maketitle    
%\vspace{.4in}
\tableofcontents
%\setstretch{1.25}

\section{Introduction} 

 In the Euclidean setting the extension problem asks us to determine the optimal range of exponents $1\leq p,r \leq \infty $ such that the following estimate holds:
$$\|(gd\sigma)^\vee\|_{L^r({\mathbb R}^d)}\leq C(p,r,d)\|g\|_{L^p(S, d\sigma)} ~~\mbox{for all}~~ g\in L^p(S, d\sigma)$$
where $d\sigma$ is a measure on the set $S$ in ${\mathbb R}^d.$  In 1967, this problem was addressed by E.M.Stein and it has been extensively studied. In particular, much attention  has been given in the case where the set $S$ is related to a hypersurface. However, this problem has not been completely solved in higher dimensions.  For a comprehensive survey of the extension problem, see \cite{Bo91}, \cite{St93}, \cite{CI98}, and \cite{Ta04} and the references therein. \\

Another interesting problem in classical harmonic analysis is the averaging problem which is to determine  the optimal range of exponents $1\leq p,r \leq \infty$ such that the following averaging estimate holds:
\begin{equation}\label{Averagingquestion} 
\|f\ast d\sigma\|_{L^r({\mathbb R}^d)}\leq C(p,r,d) \|f\|_{L^p({\mathbb R}^d)}~~\mbox{for all}~~f\in L^p({\mathbb R}^d),\end{equation}
where $d\sigma$ is a measure on a surface $S$ in ${\mathbb R}^d.$ 
This problem is originally from investigating the regularity of the fundamental solution of a wave equation at a fixed time, and many interesting results on the problem have been obtained (e.g. see \cite{St70}, \cite{St71}, \cite{Li73}, and \cite{IS96}). \\

In the finite field setting,   the extension problem and the averaging problem were recently introduced by Mockenhaupt and Tao (\cite {MT04}) and Carbery, Stones and Wright (\cite{CSW08}) respectively. In this paper we aim to develop their work by studying the topics related to an algebraic variety $$S=\{x\in \mathbb F_q^d: Q(x)=0\},$$ where $Q$ denotes a non-degenerate quadratic polynomial and $\mathbb F_q^d$ denotes the $d$-dimensional vector space over a finite field $ \mathbb F_q$ with $q$ elements. In \cite{MT04}, Mockenhaupt and Tao defined  the cone for finite fields as 
$$C_d=\{x\in \mathbb F_q^d: x_dx_{d-1}=x_1^2+\cdots+x_{d-2}^2\},$$
which is a specific form of the variety $S$. Using combinatorial  arguments, they proved that $L^2-L^4$ extension estimate holds and it actually implies the complete answer to the extension problem for the cone $C_3$ in $\mathbb F_q^3$ (see \cite{MT04}). In this paper we shall observe that the extension operator for the variety $S$ yields $L^2-L^{(2d+2)/(d-1)}$ extension estimate for all odd dimensions $d\geq 3,$ but it is not necessarily true for even dimensions $d\geq 4.$
Notice that this result recovers the sharp extension result on the cone $C_3\subset \mathbb F_q^3,$ and gives non-trivial results in higher odd dimensions. We shall also investigate  $L^p-L^r$ estimates of the averaging operator over the variety $S$ in even dimensions.

\subsection{Notation and Definition}

In order to clearly state our main results we begin by recalling some notation and definitions.  
We denote by ${\mathbb F}_q$ a finite field with $q$ elements and assume that the characteristic of ${\mathbb F}_q$ is greater than $2$,
namely $q$ is a power of an odd prime. As usual, ${\mathbb F}_q^d$ refers to the $d$-dimensional vector space over a finite field ${\mathbb F}_q$. 
Let $g: {\mathbb F}_q^d\to {\mathbb C}$ be a complex valued function on ${\mathbb F}_q^d.$ We endow the space ${\mathbb F}_q^d$ with a counting measure
$dm.$ Thus, the integral of the function $g$ over $({\mathbb F}_q^d, dm)$ is given by 
$$\int_{{\mathbb F}_q^d} g(m)~dm = \sum_{m\in {\mathbb F_{q}^d}} g(m).$$
For a fixed non-trivial additive character $\chi: {\mathbb F}_q \to {\mathbb C}$ and a complex valued function $g$ on $({\mathbb F_q^d},dm)$, we define the Fourier transform of $g$ by the following formula
\begin{equation}\label{defofrF} \widehat{g}(x)= \int_{{\mathbb F_q^d}} \chi(-m\cdot x)g(m)~dm= \sum_{m\in {\mathbb F_q^d}} \chi(-m\cdot x) g(m),\end{equation}
where $ x$ is an element in the dual space of $({\mathbb F_{q}^d},dm).$ Recall that the Fourier transform of the function $g$ on $({\mathbb F_q^d},dm)$ is actually defined on the dual space $({\mathbb F_{q}^d},dx).$ Here, we endow the dual space $({\mathbb F_{q}^d},dx)$ with a normalized counting measure $dx.$
We therefore see that if $f: ({\mathbb F_{q}^d},dx) \to {\mathbb C},$ then its integral over $({\mathbb F_{q}^d},dx) $ is given by
$$ \int_{{\mathbb F_{q}^d}} f(x)~dx= \frac{1}{q^d}\sum_{x \in {\mathbb F_{q}^d}} f(x)$$ and  the Fourier transform of the function $f$ defined on $({\mathbb F_{q}^d}, dx)$ is given by the formula
\begin{equation}\label{realFourier} \widehat{f}(m)=\int_{{\mathbb F_{q}^d}} \chi(-x\cdot m) f(x)~dx= \frac{1}{q^d} \sum_{x\in {\mathbb F_{q}^d}} \chi(x\cdot m) f(x),\end{equation}
where we recall that $m$ is any element in $({\mathbb F}_q^d,dm)$ with the counting measure $``dm",$ and we denote by $``dx"$ the normalized counting measure on $(\mathbb F_q^d, dx)$. We also recall that the Fourier inversion theorem holds:  for 
$x\in ({\mathbb F_{q}^d},dx)$ 
$$f(x)=\int_{m\in {\mathbb F_q^d}} \chi(m\cdot x) \widehat{f}(m)dm=\sum_{m\in {\mathbb F_q^d}} \chi(m\cdot x) \widehat{f}(m).$$
Using the orthogonality relation of the non-trivial additive character, that is $ \sum_{x\in \mathbb F_q^d} \chi(m\cdot x)=0$ for $m\in \mathbb F_q^d\setminus \{(0,\dots,0)\}$, we see that  the Plancherel theorem holds:
$$ \|\widehat{f}\|_{L^2({\mathbb F}_q^d, dm)}=\|f\|_{L^2({\mathbb F_{q}^d},dx)}.$$
In other words, the Plancherel theorem  yields the following formula:
\begin{equation}\label{plancherel} \sum_{m\in {\mathbb F_q^d}} |\widehat{f}(m)|^2= \frac{1}{q^d}\sum_{x\in {\mathbb F_{q}^d}} |f(x)|^2.\end{equation}
Let $f$ and $h$ be a complex valued function defined on $({\mathbb F_{q}^d},dx).$ The convolution function $f\ast h$ is defined on the space $(\mathbb F_q^d, dx)$ and it follows the rule:
$$ f\ast h (y)= \int_{x\in \mathbb F_q^d} f(y-x) h(x) dx = \frac{1}{q^d}\sum_{x\in \mathbb F_q^d} f(y-x)h(x).$$ 
It is not hard to see 
$$ \widehat{(f\ast h)}(m)= \widehat{f}(m)\cdot \widehat{h}(m) ~~\mbox{and} ~~ \widehat{(f\cdot h)}(m)= (\widehat{f}\ast \widehat{h}) (m).$$
\begin{remark} Throughout the paper we always consider the variable $``m"$ as an element of $({\mathbb F_q^d},dm)$ with the counting measure $``dm".$ On the other hand, we always use the variable $``x$ or $ y"$ to indicate  an element of $({\mathbb F_{q}^d},dx)$ with the normalized counting measure $``dx".$
Notice  from (\ref{defofrF}) and (\ref{realFourier}) that the definition of the Fourier transforms takes two different forms which depend on the domain of the Fourier transforms. 
\end{remark}
We now introduce the algebraic variety $S$ in $({\mathbb F_{q}^d}, dx)$ on which we shall work.
Given a non-degenerate quadratic polynomial $Q(x)\in \mathbb F_q[x_1,\dots,x_d]$,   we define an algebraic variety $S$ in $({\mathbb F_{q}^d},dx)$ by the set
\begin{equation}\label{defofSr} S=\{x\in {\mathbb F_{q}^d}: Q(x)=0\}.\end{equation}
By a non-singular linear substitution, any non-degenerate quadratic polynomial $Q(x)$ can be transformed into $a_1x_1^2+\cdots+a_dx_d^2 $ for some $a_j\in \mathbb F_q \setminus \{0\}, j=1,\dots,d$ (see the page 280  in \cite{LN97}). Hence, we may express the set $S$ as follows:
\begin{equation}\label{defofS}
S=\{x\in {\mathbb F_{q}^d}: a_1x_1^2+a_2x_2^2+\cdots+ a_dx_d^2=0 \}\subset ({\mathbb F_q^d},dx).
\end{equation}
We endow the set $S$ with a normalized surface measure $d\sigma$ which is given by the relation
$$ \int_{S} f(x) ~d\sigma(x) = \frac{1}{|S|}\sum_{x\in S} f(x),$$
where $|S|$ denotes the cardinality of $S.$ Note that the total mass of $S$ is one and the measure $\sigma$ is just a function on $({\mathbb F_{q}^d},dx)$ given by 
\begin{equation}\label{kohmeasure} \sigma(x)= \frac{q^d}{|S|} S(x),\end{equation}
here, and throughout the paper, we identify the set $S$ with the characteristic function on the set $S.$ For example, we write $E(x)$ for $\chi_{E}(x)$
where $E$ is a subset of ${\mathbb F_{q}^d}.$

\subsection{ Definition of extension and averaging problems for finite fields}
We recall the definition of the extension problem related to the algebraic variety $S$ in $({\mathbb F_q^d},dx).$ For $1\leq p, r \leq \infty,$ we denote by $R^*(p\to r)$ the smallest constant such that the following extension estimate holds:
$$ \|(fd\sigma)^\vee\|_{L^r({\mathbb F_q^d}, dm)} \leq R^*(p\to r) \|f\|_{L^p(S,d\sigma)}$$
for every function $f$ defined on $S$ in $({\mathbb F_{q}^d},dx).$ By duality, we  see that the quantity $R^*(p\to r)$ is also the smallest constant such that the following restriction estimate holds: for every function $g$ on $({\mathbb F_q^d},dm),$ 
\begin{equation}\label{restrictiondef}
\|\widehat{g}\|_{L^{p'}(S,d\sigma)}\leq R^*(p\to r) \|g\|_{L^{r'}({\mathbb F_q^d}, dm)},
\end{equation}
here, throughout the paper, $p'$ and $r'$ denote the dual exponents of $p$ and $r$ respectively. In other words,  $1/p+1/p'=1$ and $1/r+1/r'=1.$
The constant $R^*(p\to r)$ may depend on $q,$ the size of the underlying finite field 
${\mathbb F_{q}}.$ However, the extension problem asks to determine the exponents $1\leq p,r\leq \infty$ such that $R^*(p\to r)\lesssim 1$ where the constant in the notation $\lesssim$ is independent of $q$ and $f$. We recall that for positive numbers $A$ and $B$, the notation $A\lesssim B$ means that there exists a constant $C>0$ independent of the parameter $q$ and $f$ such that $A\leq C B.$ 
%We also review that the notation $A \lessapprox B$ is used  if for every $\varepsilon>0$ there exists $C_{\varepsilon}>0 $ such that $A\leq C_{\varepsilon} q^{\varepsilon} B.$ 
We also use the notation $A\sim B$  to illustrate that there exist $C_1>0$ and $C_2>0$ such that $ C_1 A\leq B\leq C_2A.$ 

\begin{remark}\label{endpoint}A direct calculation yields the trivial estimate, $R^*(1\to \infty)\lesssim 1.$
Using H\"{o}lder's inequality and the nesting properties of $L^p$-norms, we also see that
$$ R^*(p_1\to r)\leq R^*(p_2\to r)\quad\mbox{for}~~1\leq p_2\leq p_1\leq \infty$$
and 
$$ R^*(p\to r_1)\leq R^*(p\to r_2)\quad \mbox{for}~~ 1\leq r_2\leq r_1\leq \infty.$$
Therefore,  the optimal result  could be obtained once we find the smallest $r$ and the largest $p$ such that  $R^*(p\to r) \lesssim 1.$
\end{remark}

We now introduce the averaging problem over the algebraic variety $S$ in $({\mathbb F_q^d},dx).$ We denote by $A(p\to r)$ the smallest constant such that the following averaging estimate holds: for every $f$ defined on $({\mathbb F_{q}^d},dx)$, we have 
$$ \|f\ast d\sigma\|_{L^r({\mathbb F_{q}^d, dx})} \leq A(p\to r) \|f\|_{L^p({\mathbb F_{q}^d, dx})},$$
where $d\sigma$ is the normalized surface measure on $S$ defined as in (\ref{kohmeasure}). Like the extension problem,  the averaging problem asks to determine the  exponents $1\leq p,r \leq \infty$ such that $ A(p\to r)\lesssim 1.$ 

\section{Statement of main results}
\subsection{Results on extension problems}

As mentioned before,  Mockenhaupt and Tao (\cite{MT04})  proved that $L^2-L^4$  estimate implies the complete solution to the extension problem related to the cone in $\mathbb F_q^3.$
Using simple arguments, we modestly extends their result to higher dimensions.

\begin{theorem}\label{oddresult} Let $S$ be the variety  defined as in (\ref{defofSr}) or (\ref{defofS}). If $d\geq 3$ is odd, then we have
\begin{equation}\label{general1}R^*\left(2\to \frac{2d+2}{d-1}\right)\lesssim 1,\end{equation}
 and if $d\geq 4$ is even, then 
\begin{equation} \label{sharp22} R^*\left(2\to \frac{2d}{d-2}\right)\lesssim 1.\end{equation}
In addition, there exist specific varieties $S$ for which each result of (\ref{general1}) and (\ref{sharp22}) gives  a sharp $L^2-L^r$ extension estimate.
\end{theorem}
\begin{remark} \label{remarkgood} We shall see that in fact Theorem \ref{oddresult} is a direct result from the well-known standard Tomas-Stein type argument.
However, the conclusions of Theorem \ref{oddresult} are very interesting, in part because they are  inconsistent with the facts in the Euclidean case. For example, if $S\subset \mathbb R^d$ is a compact subset of the cone, then it is well known that the $L^2-L^{2d/(d-2)}$ estimate gives the sharp $L^2-L^r$ extension estimate for all dimension $d\geq 3$ (see \cite{Ta04} or \cite{Wo01}).  Notice that  the conclusion (\ref{general1}) is much better than that in the Euclidean case although the conclusion (\ref{sharp22}) in even dimensions is exactly the same.
In the Euclidean setting, the curvature on the surface makes an important role in determining the extension estimates. 
On the other hand, the extension estimates for finite fields can be determined in accordance with the maximal size of affine subspaces in the surface $S.$ This explains why the result (\ref{general1}) for odd dimensions is much better than (\ref{sharp22}) for even dimensions. In fact, the surface $S$ in even dimensions may contain a $d/2$-dimensional subspace but this never happens in odd dimensions, because $d/2$ is not an integer for odd $d$. The conclusion (\ref{general1}) shows that  if $d\geq 3$ is odd, then $q^{(d-1)/2}$  is the maximal  cardinality of  subspaces contained in the variety $S.$ 

\end{remark}

\subsection{Results on averaging problems}
\begin{theorem}\label{mainpaper} Let $S$ be the algebraic variety in $({\mathbb F_{q}^d},dx)$ defined as in (\ref{defofSr}) or (\ref{defofS}). If $d\geq 3$ is odd, then we have
\begin{equation}\label{sharpodda} A(p\to r)\lesssim 1 ~~\iff ~~ \left(\frac{1}{p}, \frac{1}{r}\right) \in {\mathbb T},\end{equation}
where ${\mathbb T}$ denotes the convex hull of points $(0,0), (0,1),(1,1)$ and $(d/(d+1), 1/(d+1)).$
On the other hand, if $d\geq 4$ is even, then
\begin{equation}\label{sharpupto}A(p\to r)\lesssim  1 \quad \mbox{for}~~ \left(\frac{1}{p}, \frac{1}{r}\right) \in \Omega\setminus\{P_1,P_2\} \end{equation}
where $\Omega$ denotes  the convex hull of points $(0,0), (0,1), (1,1),$
$$P_1= \left(\frac{d^2-2d+2}{d(d-1)}, \frac{1}{(d-1)}  \right)~~\mbox{and}~~ P_2=\left(\frac{d-2}{d-1}, \frac{d-2}{d(d-1)}  \right).$$
In addition, if $d\geq 4$ is even and $P_1=(1/p, 1/r)$ then the restricted strong-type estimate 
\begin{equation}\label{sharpupto1} \|f\ast d\sigma\|_{L^r(\mathbb F_q^d, dx)}\lesssim \|f\|_{L^{p,1}(\mathbb F_q^d, dx)}\end{equation}
holds, and if $d\geq 4$ is even and $P_2=(1/r^\prime, 1/p^\prime)$ then the weak-type estimate 
\begin{equation}\label{sharpupto2} \|f\ast d\sigma\|_{L^{p^\prime, \infty}(\mathbb F_q^d, dx)} \lesssim \|f\|_{L^{r^\prime}(\mathbb F_q^d, dx)}\end{equation}
holds.
Finally,  the averaging results in even dimensions  are sharp in the sense that  if $(1/p, 1/r) \notin \Omega$ and $S$ contains a $d/2$-dimensional subspace, then  $L^p-L^r$ averaging estimate is impossible. 
\end{theorem}

The results in  Theorem \ref{mainpaper} are also interesting since it contrasts with well-known facts in the Euclidean case.
In the Euclidean space it is well known that  if a hypersurface has everywhere non-vanishing Gaussian curvature , then  $L^p-L^r$ averaging estimate holds if and only if $(1/p, 1/r)$ lies in the triangle with vertices $(0,0), (1,1), $ and $(d/(d+1), 1/(d+1) ).$\footnote{ If $1\leq r<p\leq \infty,$ then the $L^p-L^r$ averaging estimate is impossible in the Euclidean case but it always holds in the finite field setting. Therefore, it would be only interesting to find  the difference   in the case when $1\leq p\leq r\leq \infty$.}
 However, if the Gaussian curvature is allowed to vanish, then the averaging estimates are getting worse (see \cite{Li73}, \cite{St70}, and \cite{St71}) .  For example, since it is clear that $S=\{x\in\mathbb R^d: x_1^2+x_2^2+ \dots + x_d^2=0\}$ has everywhere  vanishing Gaussian curvature away from the origin, the averaging estimates in the Euclidean case must be much worse than our result (\ref{sharpodda}) in the finite field case.  The other interesting point  of Theorem \ref{mainpaper} says that  the sharp averaging estimates (\ref{sharpodda}) in odd dimensions are better than those in even dimensions.  The main reason for the difference is the same as what  has been mentioned  in Remark \ref{remarkgood}.
Since the variety $S$ in odd dimensions $d\geq 3$ can only contain  a subspace $H$ with  the cardinality  at most $q^{(d-1)/2}$. We shall see  that the Fourier transform of the surface measure $d\sigma$ yields a good decay estimate that  the sharp averaging estimates (\ref{sharpodda}) can be directly obtained by the well-established Euclidean arguments. On the other hand, if  the dimension $d\geq 4$ is even, then a relatively big subspace $H$ with the cardinality $q^{d/2}$ may lie in $S.$ In this case, the averaging problem becomes much harder but we still can obtain relatively good results by applying our extension result (\ref{sharp22}).

\subsection{Outline of the remaining parts of the paper}
In Section \ref{three}, we summarize the necessary conditions for $R^*(p\to r)\lesssim 1$ and $A(p\to r)\lesssim 1.$  In Section \ref{four}, we compute the explicit form of the Fourier transform on the variety $S,$ which makes a crucial role in the proofs of our results. In Section \ref{five}, the proof of Theorem \ref{oddresult} is given. In the last section, we complete the proof of Theorem \ref{mainpaper}.

 \section{Necessary conditions for $L^p-L^r$ extension and averaging estimates}\label{three}
 In this section, we review the necessary conditions for $R^*(p\to r)\lesssim 1$ and $A(p\to r)\lesssim 1.$
 Mockenhaupt and Tao (\cite{MT04}) introduced the necessary conditions for $L^p-L^r$ extension estimates related to the cone 
 $C_3=\{x\in \mathbb F_q^3: x_2x_3=x_1^2\}$ and they proved that the necessary conditions are in fact sufficient. Based on the similar arguments as in \cite {MT04}, it is not hard to find the necessary conditions for the case of higher dimensions. Here, we state the necessary conditions for $L^p-L^r$ extension estimates related to the variety $S=\{x\in \mathbb F_q^d: a_1x_1^2+\dots+a_dx_d^2=0\}, d\geq 3,$  and we leave the proof to the readers.
 \begin{lemma}\label{Necone} If $d\geq 4$ is even and $S$ contains a $d/2-$dimensional subspace, then the necessary conditions for $R^*(p\to r)\lesssim 1$ take the followings:
 $$ r\geq \frac{2d-2}{d-2} \quad \mbox{and}\quad r\geq \frac{dp}{(d-2)(p-1)}.$$
 On the other hand, if $d\geq 3$ is odd, $S$ contains a $(d-1)/2-$dimensional subspace, and $-a_ia_j^{-1}$ is a square number for some $i,j=1,2,\cdots, d$ with $i\neq j$, then the necessary conditions for $R^*(p\to r)\lesssim 1$ are given by the relation:
 $$ r\geq \frac{2d-2}{d-2} \quad \mbox{and}\quad r\geq \frac{(d+1)p}{(d-1)(p-1)}.$$\end{lemma}
 
 The necessary conditions for $L^p-L^r$ averaging estimates are well-known by Carbery, Stones and Wright (\cite{CSW08}). In our case, the necessary conditions can be stated as follows:
 \begin{lemma}\label{neave}
 For  $a_j\neq 0, j=1,\dots,d,$ let $S=\{x\in {\mathbb F_{q}^d}: a_1x_1^2+ \cdots+a_dx_d^2=0\}.$
Then $A(p\to r)\lesssim 1$ only if $(1/p, 1/r)$ lies in the convex hull of the points 
\begin{equation}\label{N4}(0,0),(0,1),(1,1), ~\mbox{and}~ \left(\frac{d}{d+1}, \frac{1}{d+1}\right).\end{equation}
Moreover, if $d\geq 4$ is even and $S$ contains a $d/2-$dimensional affine subspace $H$, then $A(p\to r)\lesssim 1$ only if $(1/p, 1/r)$ lies in the convex hull of the points $(0,0),(0,1),(1,1),$ 
\begin{equation}\label{N5} \left( \frac{d^2-2d+2}{d(d-1)}, \frac{1}{d-1}\right), ~\mbox{and}~\left(\frac{d-2}{d-1},\frac{d-2}{d(d-1)} \right).\end{equation}
 \end{lemma}
 
 \section{The Fourier transform of the surface measure $d\sigma$}\label{four}
 In this section we obtain an explicit formula for the Fourier transform of the surface measure $d\sigma$ on the surface $S$ defined as in (\ref{defofS}). We shall see that the Fourier transform is closely related to the classical Gauss sums. Moreover, it makes a key role to prove our main results on both the extension problem and  the averaging problem. It is useful to review classical Gauss sums in the finite field setting.
In the remainder of this paper, we fix the additive character $\chi$ as a canonical additive character of ${\mathbb F_q}$ and $\eta$ always denotes  the quadratic character of ${\mathbb F_q}.$ Recall that $\eta(t)=1$ if $s$ is a square number in ${\mathbb F}_q\setminus \{0\}$ and $\eta(t)=-1$ if $t$ is not a square number in ${\mathbb F_q}\setminus \{0\}.$  We also recall that $\eta(0)=0, \eta^2\equiv 1, \eta(ab)=\eta(a)\eta(b)$ for $a,b\in {\mathbb F_q},$ and $\eta(t)=\eta(t^{-1})$ for $t\neq 0.$ For each $t\in {\mathbb F_q},$ the Gauss sum $G_t(\eta,\chi)$ is defined by 
$$ G_t(\eta,\chi)=\sum_{s\in {\mathbb F_q}\setminus \{0\}} \eta(s)\chi(ts).$$
The absolute value of the Gauss sum is given by the relation
$$|G_t(\eta,\chi)|=\left\{\begin{array}{ll} q^{\frac{1}{2}} \quad &\mbox{if}~~t\neq 0\\
0\quad&\mbox{if} ~~t=0.\end{array}\right.$$
In addition, we have the following formula 
\begin{equation}\label{complete}\sum_{s\in {\mathbb F_q}} \chi(ts^2)=\eta(t)G_1(\eta,\chi)~~\mbox{for any}~~ t\neq 0.\end{equation}
 For the nice proofs for the properties related to the Gauss sums, see 
Chapter $5$ in \cite{LN97} and Chapter $11$ in \cite{IKow04}. When we complete the square and apply a change of variable,  the formula (\ref{complete}) yields the following equation: for each $a\in {\mathbb F_q}\setminus \{0\}, b \in {\mathbb F_q}$
\begin{equation}\label{method} \sum_{s\in {\mathbb F_q}} \chi(as^2+bs)= G_1(\eta,\chi) \eta(a)\chi\left(\frac{b^2}{-4a}\right).\end{equation}
We shall name the skill used to obtain the formula (\ref{method}) as the complete square method. 
Relating the inverse Fourier transform of $d\sigma$ with the Gauss sum, we shall obtain an explicit form of $(d\sigma)^\vee$, the inverse Fourier transform of the surface measure on $S.$ We have the following lemma.
\begin{lemma}\label{explicit}
Let $d\sigma$ be the surface measure on $S$ defined as in (\ref{defofS}). If $d\geq 3$ is odd, then we have

$$ (d\sigma)^\vee(m)=\left\{\begin{array}{ll} q^{d-1} |S|^{-1} \quad &\mbox{if}~~ m=(0,\dots,0)\\
0 \quad &\mbox{if}~~ m\neq (0,\dots,0),~ \frac{m_1^2}{a_1}+\cdots+ \frac{m_d^2}{a_d}=0\\
 \frac{G_1^{d+1}}{q|S|} \eta(-a_1 \cdots  a_d) \eta\left(\frac{m_1^2}{a_1}+\cdots+\frac{m_d^2}{a_d}\right) \quad &\mbox{if} ~~ \frac{m_1^2}{a_1}+\cdots+ \frac{m_d^2}{a_d}\neq 0. \end{array}\right.$$
 If $d\geq 2$ is even, then we have
 $$ (d\sigma)^\vee(m)=\left\{ \begin{array}{ll} q^{d-1} |S|^{-1}+ \frac{G_1^d}{|S|} (1-q^{-1})\eta(a_1\cdots a_d) \quad &\mbox{if}~~ m=(0,\dots,0)\\
\frac{G_1^d}{|S|}(1-q^{-1})\eta(a_1\cdots a_d)  \quad &\mbox{if}~~ m\neq (0,\dots,0),~ \frac{m_1^2}{a_1}+\cdots+ \frac{m_d^2}{a_d}=0\\
 -\frac{G_1^{d}}{q|S|} \eta(a_1 \cdots  a_d)  \quad &\mbox{if} ~~ \frac{m_1^2}{a_1}+\cdots+ \frac{m_d^2}{a_d}\neq 0, \end{array}\right.$$
 here, and throughout this paper, we write $G_1$ for the Gauss sum $G_1(\eta,\xi)$ and $\eta$ denotes the quadratic character of ${\mathbb F_q}.$
\end{lemma}

\begin{proof}
Using the definition of the inverse Fourier transform and the orthogonality relations of the nontrivial additive character $\chi$ of ${\mathbb F_q}$, we see
\begin{align*}
(d\sigma)^\vee(m)&= |S|^{-1}\sum_{x\in S} \chi(x\cdot m)\\
&=|S|^{-1} q^{-1}\sum_{x\in {\mathbb F_q^d}}\sum_{s\in {\mathbb F_q}} \chi\left( s(a_1x_1^2+\cdots+ a_dx_d^2)\right)~\chi(x\cdot m)\\
&= q^{d-1}|S|^{-1} \delta_0(m)+ |S|^{-1} q^{-1}\sum_{x\in {\mathbb F_q^d}}\sum_{s\neq 0} \chi\left( s(a_1x_1^2+\cdots+ a_dx_d^2)\right)~\chi(x\cdot m)\\
&=q^{d-1}|S|^{-1} \delta_0(m)+ |S|^{-1}q^{-1} \sum_{s\neq 0} \prod_{j=1}^d \sum_{x_j\in {\mathbb F_q}} \chi(sa_j x_j^2+m_jx_j).
\end{align*}
Use the complete square method (\ref{method}), compute the sums over $x_j\in {\mathbb F_q}$ and then obtain that
$$ (d\sigma)^\vee(m)=q^{d-1}|S|^{-1}\delta_0(m)+ G_1^d |S|^{-1}q^{-1} \eta(a_1\cdots a_d) \sum_{s\neq 0} \eta^d(s) \chi\left( -\frac{1}{4s}\left(\frac{m_1^2}{a_1}+\cdots+\frac{m_d^2}{a_d}\right)\right).$$
{\bf Case I}. Suppose that $d\geq 3$ is odd. Then $ \eta^d\equiv \eta$, because $\eta$ is the multiplicative character of order two. Therefore, if 
$\frac{m_1^2}{a_1}+\cdots+\frac{m_d^2}{a_d}=0$, then the proof is complete, because  $\sum_{s\in {\mathbb F_q}\setminus \{0\}}\eta(s)=0.$ On the other hand, if $ \frac{m_1^2}{a_1}+\cdots+\frac{m_d^2}{a_d}\neq 0,$ then the statement follows from using  a change of variable,$-\frac{1}{4s}\left(\frac{m_1^2}{a_1}+\cdots+\frac{m_d^2}{a_d}\right) \to s,$  and the facts that $\eta(4)=1, \eta(s)=\eta(s^{-1})$ for $s\neq 0$, and $ G_1=\sum_{s\neq 0} \eta(s)\chi(s).$ \\
{\bf Case II}. Suppose that $ d\geq 2$ is even. Then $\eta^{d}\equiv 1.$ The proof is complete, because $\sum_{s\neq 0} \chi(as)=-1$ for all $a\neq 0$, and $\sum_{s\neq 0} \chi(as)=(q-1)$ if $a=0.$
\end{proof}

 Lemma \ref{explicit} yields the following corollary. 

\begin{corollary}\label{decay} If $d\geq 3$ is odd, then it follows that
\begin{equation}\label{odddecay}\begin{array}{ll}(d\sigma)^\vee(0,\dots,0)= 1, \quad & \\
|(d\sigma)^\vee(m)|\lesssim q^{-\frac{(d-1)}{2}}\quad &\mbox{if}~~ m\neq (0,\dots,0), \end{array}\end{equation}

and if $d\geq 4$ is even, then we have
\begin{equation}\label{evendecay}\begin{array}{ll} (d\sigma)^\vee(0,\dots,0)= 1, \quad &\\
|(d\sigma)^\vee(m)|\lesssim q^{-\frac{(d-2)}{2}} \quad &\mbox{if}~~m\neq (0,\dots,0). \end{array}\end{equation}
\end{corollary} 
\begin{proof}
Recall that  the Fourier inverse transform of the surface measure $d\sigma$ is given by the relation
$$ (d\sigma)^\vee(m)=\int_{S} \chi(x\cdot m) d\sigma=\frac{1}{|S|}\sum_{x\in S} \chi(x\cdot m)$$
where $m\in ({\mathbb F_q^d},dm).$ Therefore, it is clear that $ (d\sigma)^\vee(0,\ldots,0)=1$ for all $d\geq 2.$
If we compare this with the values $(d\sigma)^\vee(0,\ldots,0)$ given by Lemma \ref{explicit}, then we see that 
$|S|\sim q^{d-1}$ for $d\geq 3.$ Since the absolute of the Gauss sum $G_1$ is exactly $q^{1/2}$, the statements in Corollary \ref{decay} follows immediately from Lemma \ref{explicit}.
\end{proof}

\section{Proof of Theorem \ref{oddresult} (Extension Theorems)}\label{five}
We begin by proving the last statement in Theorem \ref{oddresult}.
We choose  a variety $S$ with $a_j=1$ for $j$ odd and $a_j=-1$ otherwise. It follows that
if $d\geq 3$ is odd, then the variety $S$ contains the $(d-1)/2-$dimensional subspace
$$H=\left\{ (t_1,t_1, \dots, t_j,t_j, \dots, t_{(d-1)/2}, t_{(d-1)/2}, 0): t_k\in \mathbb F_q^d, k=1,2,\dots, (d-1)/2\right\},$$ 
and if $d\geq 4$ is even, then it contains the $d/2-$dimensional subspace
$$W=\left\{ (t_1,t_1, \dots, t_j,t_j, \dots, t_{d/2}, t_{d/2}): t_k\in \mathbb F_q^d, k=1,2,\dots, d/2\right\}.$$ 
Thus, the last statement in Theorem \ref{oddresult} follows immediately from the necessary conditions in Lemma \ref{Necone}.
 Next, observe that  the statements of (\ref{general1}) and (\ref{sharp22}) follow from Corollary \ref{decay} and the following lemma 
 which can be proved by a routine modification of the Euclidean Tomas-Stein type argument.

\begin{lemma}\label{Tao}
Let $d\sigma$ be the surface measure on the algebraic variety $S \subset ({\mathbb F_{q}^d},dx)$ defined as in (\ref{defofS}).
If $|(d\sigma)^\vee(m)|\lesssim q^{-\frac{\alpha}{2}}$ for some $\alpha>0$ and for all $m\in {\mathbb F_q^d}\setminus (0,\dots,0),$ then we have 
$$ R^*\left(2\to \frac{2(\alpha+2)}{\alpha}\right)\lesssim 1.$$
\end{lemma}
\begin{proof}
By duality, it suffices to prove that the following restriction estimate holds: for every function $g$ defined on $({\mathbb F_q^d},dm),$ we have
$$ \|\widehat{g}\|^2_{L^2(S,d\sigma)}\lesssim \|g\|^2_{L^{\frac{2(\alpha+2)}{\alpha+4}}({\mathbb F_q^d}, dm)}.$$
By the orthogonality principle and H\"{o}lder's inequality, we see 
$$ \|\widehat{g}\|^2_{L^2(S,d\sigma)} \leq \|g\ast (d\sigma)^\vee \|_{L^{\frac{2(\alpha +2)}{\alpha}}({\mathbb F_q^d}, dm)}\|g\|_{L^{\frac{2(\alpha +2)}{\alpha+4}}({\mathbb F_q^d}, dm)}.$$
It therefore suffices to show that for every function $g$ on $({\mathbb F_q^d},dm),$ 
$$\|g\ast (d\sigma)^\vee \|_{L^{\frac{2(\alpha +2)}{\alpha}}({\mathbb F_q^d}, dm)}\lesssim \|g\|_{L^{\frac{2(\alpha +2)}{\alpha+4}}({\mathbb F_q^d}, dm)}.$$
Define $ K=(d\sigma)^\vee -\delta_0.$ Since $(d\sigma)^\vee(0,\dots,0)=1$, we see that $ K(m)=0$ if $m=(0,\dots,0)$, and $K(m)=(d\sigma)^\vee(m)$ if $m\in {\mathbb F_q^d}\setminus \{(0,\dots,0)\}.$ 
It follows that
$$\begin{array}{ll} \|g\ast \delta_0 \|_{L^{\frac{2(\alpha +2)}{\alpha}}({\mathbb F_q^d}, dm)}&=\|g \|_{L^{\frac{2(\alpha +2)}{\alpha}}({\mathbb F_q^d}, dm)}\\
&\leq \|g\|_{L^{\frac{2(\alpha +2)}{\alpha+4}}({\mathbb F_q^d}, dm)},\end{array}$$
where the inequality follows from the fact that $dm$ is the counting measure and $\frac{2(\alpha +2)}{\alpha}\geq \frac{2(\alpha +2)}{\alpha+4}.$
Thus, it is enough to show that for every $g$ on $({\mathbb F_q^d},dm),$
\begin{equation}\label{kernel}
\|g\ast K \|_{L^{\frac{2(\alpha +2)}{\alpha}}({\mathbb F_q^d}, dm)}\lesssim \|g\|_{L^{\frac{2(\alpha +2)}{\alpha+4}}({\mathbb F_q^d}, dm)}.\end{equation}
We now claim that the following two estimates hold: for every function $g$ on $({\mathbb F_q^d},dm),$
\begin{equation}\label{ltwo} \|g\ast K\|_{L^2({\mathbb F_q^d}, dm)}\lesssim q \|g\|_{L^2({\mathbb F_q^d},dm)}
\end{equation}
and 
\begin{equation}\label{linfty}\|g\ast K\|_{L^\infty({\mathbb F_q^d}, dm)}\lesssim q^{-\frac{\alpha}{2}} \|g\|_{L^1({\mathbb F_q^d},dm)}.
\end{equation}
Note that the estimate (\ref{kernel}) follows by interpolating (\ref{ltwo}) and (\ref{linfty}). It therefore remains to show that both (\ref{ltwo}) and (\ref{linfty}) hold. Using Plancherel, the inequality (\ref{ltwo}) follows from the following observation:
$$ \begin{array}{ll} \|g\ast K\|_{L^2({\mathbb F_q^d}, dm)}&=\|\widehat{g}\widehat{K}\|_{L^2({\mathbb F_q^d}, dx)}\\
&\leq \|\widehat{K}\|_{L^\infty({\mathbb F_q^d}, dx)} \|\widehat{g}\|_{L^2({\mathbb F_q^d}, dx)}\\
&\lesssim q\|g\|_{L^2({\mathbb F_q^d}, dm)},\end{array}$$
where the last line is due to the observation that for each $x\in ({\mathbb F_q^d},dx)$\\
$\widehat{K}(x)= d\sigma(x)-\widehat{\delta_0}(x)=q^d|S|^{-1} S(x)-1 \lesssim q.$ On the other hand, the estimate (\ref{linfty}) follows from Young's inequality and the assumption on the Fourier decay estimates away from the origin. Thus, the proof is complete.
\end{proof}

\section{Proof of Theorem \ref{mainpaper} (Averaging Theorems)}

\subsection{Proof of (\ref{sharpodda}) in Theorem \ref{mainpaper}}
Because of the necessary condition (\ref{N4}) in Theorem \ref{neave}, it suffices to prove that if $(1/p,1/r)\in {\mathbb T}$, then $A(p\to r)\lesssim 1,$ where ${\mathbb T}$ is the convex hull of points $(0,0),(0,1),(1,1)$, and $(d/(d+1),1/(d+1)).$
Since both $d\sigma$ and $({\mathbb F_{q}^d},dx)$ have total mass $1$ it is clear that if $1\leq r \leq p\leq \infty,$ then
\begin{equation}\label{K1} \|f\ast d\sigma\|_{L^r({\mathbb F_{q}^d, dx})} \leq \|f\|_{L^p({\mathbb F_{q}^d, dx})}.\end{equation}
Using the interpolation theorem, it is enough to prove that 
$$ A\left((d+1)/d\to d+1\right)\lesssim 1.$$
Since the dimension $d\geq 3$ is odd, it follows from the first part of Corollary \ref{decay} that
$$|(d\sigma)^\vee(m)|\lesssim q^{-\frac{(d-1)}{2}}\quad \mbox{if}~~ m\neq (0,\dots,0),$$
and we complete the proof by using the lemma below due to the authors in \cite{CSW08}.
\begin{lemma}\label{Carbery}
Let $d\sigma$ be the surface measure on the algebraic variety $S \subset ({\mathbb F_{q}^d},dx)$ defined as in (\ref{defofSr}).
If $|(d\sigma)^\vee(m)|\lesssim q^{-\frac{\alpha}{2}}$ for all $m\in {\mathbb F_q^d}\setminus (0,\dots,0)$ and for some $\alpha>0,$ then we have 
$$ A\left(\frac{\alpha+2}{\alpha+1}\to \alpha+2\right)\lesssim 1.$$
\end{lemma}
\begin{proof} Consider a function $K$ on $({\mathbb F_q^d}, dm)$ defined as $K=(d\sigma)^\vee -\delta_0.$
We want to prove that for every function $f$ on $({\mathbb F_q^d},dx)$, 
$$\|f\ast d\sigma\|_{L^{\alpha+2}({\mathbb F_q^d},dx)}\lesssim \|f\|_{L^{\frac{\alpha+2}{\alpha+1}}({\mathbb F_q^d},dx)}.$$
Since $d\sigma= \widehat{K}+\widehat{\delta_0}= \widehat{K}+1$ and $\|f\ast 1\|_{L^{\alpha+2}({\mathbb F_q^d},dx)}\lesssim \|f\|_{L^{\frac{\alpha+2}{\alpha+1}}({\mathbb F_q^d},dx)}$, it suffices to show that for every $f$ on $({\mathbb F_q^d},dx),$
\begin{equation}
\|f\ast \widehat{K}\|_{L^{\alpha+2}({\mathbb F_q^d},dx)}\lesssim \|f\|_{L^{\frac{\alpha+2}{\alpha+1}}({\mathbb F_q^d},dx)}.
\end{equation}
Notice that this  can be done by interpolating the following two estimates:
\begin{equation}\label{first1}
\|f\ast \widehat{K}\|_{L^2({\mathbb F_q^d},dx)}\lesssim  q^{-\frac{\alpha}{2}}\|f\|_{L^2({\mathbb F_q^d},dx)}
\end{equation}
and 
\begin{equation}\label{second2}
\|f\ast \widehat{K}\|_{L^{\infty}({\mathbb F_q^d},dx)}\lesssim  q \|f\|_{L^1({\mathbb F_q^d},dx)}.
\end{equation}
The inequality (\ref{first1}) follows from  the Plancherel theorem, the size assumption of $|(d\sigma)^\vee|$, and the definition of $K.$
On the other hand, the inequality (\ref{second2}) follows from Young's inequality and the observation that $\|\widehat{K}\|_{L^\infty({\mathbb F_q^d},dx)}\lesssim q.$ Thus, the proof of Lemma \ref{Carbery} is complete. 
\end{proof}

\subsection{Proof of  Theorem \ref{mainpaper} in the case of even dimensions} 
First, observe that  the statement for the sharpness  follows from the necessary condition (\ref{N5}). 
Also recall from (\ref{K1}) that $A(p\to r)\lesssim 1$ for $1\leq q\leq p\leq \infty.$ 

It is clear by duality that the statement (\ref{sharpupto1}) implies the statement (\ref{sharpupto2}).
By the interpolation theorem, we also see that  the statements of (\ref{sharpupto1}) and (\ref{sharpupto2}) imply the statement (\ref{sharpupto}).
Therefore, it suffices to prove the restricted strong-type estimate (\ref{sharpupto1}).
More precisely,  it amounts to showing
\begin{equation}\label{resweak}\|E\ast d\sigma\|_{L^{d-1}({\mathbb F_q^d},dx)} \lesssim  \|E\|_{L^{\frac{d(d-1)}{d^2-2d+2}}({\mathbb F_q^d},dx)} \quad \mbox{for all}~~E \subset  \mathbb F_q^d.\end{equation}
We now consider the Bochner-Riesz kernel $K$ on $(\mathbb F_q^d,dm)$ defined by $K=(d\sigma)^\vee -\delta_0,$
where $\delta_0(m)=1$ if $m=(0,\dots,0),$ and $0$ otherwise.
Our task is to establish the following two inequalities: for all $E\subset \mathbb F_q^d,$
\begin{equation}\label{aa} \|E\ast \widehat{\delta_0}\|_{L^{d-1}({\mathbb F_q^d},dx)} \lesssim  \|E\|_{L^{\frac{d(d-1)}{d^2-2d+2}}({\mathbb F_q^d},dx)} \quad \mbox{for all}~~E \subset  \mathbb F_q^d
\end{equation}
and
\begin{equation}\label{bb}\|E\ast \widehat{K}\|_{L^{d-1}({\mathbb F_q^d},dx)} \lesssim  \|E\|_{L^{\frac{d(d-1)}{d^2-2d+2}}({\mathbb F_q^d},dx)} \quad \mbox{for all}~~E \subset  \mathbb F_q^d.
\end{equation}
Since $\widehat{\delta_0}=1$ and the total mass of $\mathbb F_q^d$ is one, the inequality (\ref{aa}) follows immediately from Young's inequality for convolution. On the other hand,  the inequality (\ref{bb}) can be obtained by interpolating the following two inequalities:
for all $E\subset \mathbb F_q^d,$
\begin{equation}\label{in1}
\|E\ast \widehat{K}\|_{L^\infty({\mathbb F_q^d}, dx)}\lesssim q \|E\|_{L^1(\mathbb F_q^d, dx)}
\end{equation}
and 
\begin{equation}\label{in2}
\|E\ast \widehat{K}\|_{L^2({\mathbb F_q^d}, dx)}\lesssim q^{\frac{-d+3}{2}} \|E\|_{L^\frac{2d}{d+2}(\mathbb F_q^d,dx)}.
\end{equation}
Since the inequality (\ref{in1})  follows immediately from Young's inequality and the observation that $\|\widehat{K}\|_{L^\infty(\mathbb F_q^d,dx)} \lesssim q,$ it remains to prove that (\ref{in2}) holds. Namely, we must show that 
$$\|E\ast \widehat{K}\|_{L^2({\mathbb F_q^d}, dx)}\lesssim q^{-d+\frac{1}{2}} |E|^{\frac{d+2}{2d}}\quad \mbox{for all}~~ E\subset \mathbb F_q^d.$$
It suffices to prove the following inequality which gives the better estimate in the case when $1\leq |E|\leq q^{\frac{d}{2}}:$ 

\begin{equation}\label{cool}\|E\ast \widehat{K}\|_{L^2({\mathbb F_q^d},dx)} \lesssim \left\{\begin{array}{ll} q^{-d+\frac{1}{2}} |E|^{\frac{d+2}{2d}}\quad &\mbox{if}~~ 1\leq |E|\leq q^{\frac{d}{2}}\\
q^{-d+1}|E|^{\frac{1}{2}} \quad &\mbox{if}~~ q^{\frac{d}{2}}\leq |E|\leq q^d. \end{array}\right. \end{equation}

Using the Plancherel theorem, we have
\begin{align*} \|E\ast \widehat{K}\|^2_{L^2({\mathbb F_q^d},dx)}&=\|\widehat{E} K\|^2_{L^2({\mathbb F_q^d},dm)}\\
&=\sum_{m\in {\mathbb F_q^d}} |\widehat{E}(m)|^2 |K(m)|^2=\sum_{m\neq (0,\dots,0)} |\widehat{E}(m)|^2 |(d\sigma)^\vee(m)|^2,
\end{align*}
where the last line follows from the definition of $K$ and the fact that $(d\sigma)^\vee(0,\dots,0)=1.$
Since  $|S|\sim q^{d-1}, |\eta|\equiv 1,$ and the absolute value of the Gauss sum $G_1$ is $q^{1/2},$  using the explicit formula for $(d\sigma)^\vee$ in the second part of Lemma \ref{explicit} shows that
$$\|E\ast \widehat{K}\|^2_{L^2({\mathbb F_q^d},dx)} $$
$$\sim \frac{1}{q^{d-2}} \sum_{\substack{m\neq (0,\dots,0):\\ \frac{m_1^2}{a_1}+\cdots+ \frac{m_d^2}{a_d}=0}} |\widehat{E}(m)|^2 +\frac{1}{q^d}\sum_{\substack{m\neq (0,\dots,0):\\ \frac{m_1^2}{a_1}+\cdots+ \frac{m_d^2}{a_d}\neq 0}} |\widehat{E}(m)|^2 =\mbox{I} +\mbox{II}.$$
From the Plancherel theorem  (\ref{plancherel}), we see that
$$\mbox{II}\leq \frac{1}{q^d} \sum_{m\in {\mathbb F_q^d}} |\widehat{E}(m)|^2 = q^{-2d}|E|.$$
We claim that the upper bound of $\mbox{I}$ is given by 
\begin{equation}\label{claim} \mbox{I}\lesssim  \min\left\{ q^{-2d+1}|E|^{\frac{d+2}{d}}, q^{-2d+2}|E|\right\},\end{equation}
which shall be proved  later.
It follows that \begin{align*} \|E\ast \widehat{K}\|^2_{L^2({\mathbb F_q^d},dx)} &\lesssim \min\left\{ q^{-2d+1}|E|^{\frac{d+2}{d}}, q^{-2d+2}|E|\right\} + q^{-2d}|E|\\
&\sim \min\left\{ q^{-2d+1}|E|^{\frac{d+2}{d}}, q^{-2d+2}|E|\right\}.\end{align*}
By a direct calculation, we see that this estimate implies (\ref{cool}). 
Thus, our last work is to prove the claim (\ref{claim}). Notice that (\ref{claim}) can be obtained by using the following lemma based on the dual extension theorem.
\begin{lemma}
For any subset $E$ of $({\mathbb F_{q}^d},dx)$  and $b_j\neq 0$ for $j=1,\dots, d,$ if $d\geq 4$ is even, then we have
$$ \sum_{m\in S} |\widehat{E}(m)|^2 := \sum_{m\in S} \left| q^{-d} \sum_{x\in E} \chi(-m\cdot x)\right|^2 \lesssim \min\left\{ q^{-(d+1)}|E|^{\frac{d+2}{d}} ,~ q^{-d}|E| \right\},$$
where $S=\{ m\in {\mathbb F_q^d}: b_1m_1^2+\cdots+ b_dm_d^2=0\}\subset ({\mathbb F_q^d},dm)$
\end{lemma}
\begin{proof} It is clear from the Plancherel theorem  that 
$$\sum_{m\in S} |\widehat{E}(m)|^2 \leq \sum_{m\in {\mathbb F_q^d}} |\widehat{E}(m)|^2 = q^{-d}|E|.$$ 
It therefore remains to show that
\begin{equation}\label{finalcor} \sum_{m\in S} |\widehat{E}(m)|^2 := \sum_{m\in S} \left| q^{-d} \sum_{x\in E} \chi(-m\cdot x)\right|^2  \lesssim q^{-(d+1)}|E|^{\frac{d+2}{d}}.\end{equation}
Since the space $({\mathbb F_q^d},dx)$ is isomorphic to its dual space $({\mathbb F_{q}^d},dm)$ as an abstract group, we may identify  the space $({\mathbb F_q^d},dx)$ with the dual space $({\mathbb F_{q}^d},dm).$ Thus, they possess same algebraic structures. Recall that we have endowed them with different measures: the counting measure $dm$ for $ ({\mathbb F_q^d},dm)$ and the normalized counting measure $dx$ for $({\mathbb F_{q}^d},dx).$ For these reasons, the inequality (\ref{finalcor}) is essentially same as the following: for every subset $E$ of $({\mathbb F_{q}^d},dm)$
\begin{equation}\label{aimcor}\sum_{x\in S} q^{-2d}|\widehat{E}(x)|^2\lesssim q^{-(d+1)}|E|^{\frac{d+2}{d}},\end{equation}
where $S$ is considered as  
$$S=\{ x\in {\mathbb F_{q}^d}: b_1x_1^2+\cdots+ b_dx_d^2=0\}\subset ({\mathbb F_{q}^d},dx)\quad\mbox{and}~~ \widehat{E}(x)=\sum_{m\in \mathbb F_q^d} \chi(-m\cdot x) E(m). $$
By duality  (\ref{restrictiondef}), the statement (\ref{sharp22}) in Theorem \ref{oddresult} implies that the following restriction estimate holds: for every function $g$ on $({\mathbb F_q^d},dm),$ 
$$\|\widehat{g}\|^2_{L^{2}(S,d\sigma)}\lesssim \|g\|^2_{L^{\frac{2d}{d+2}}({\mathbb F_q^d}, dm)}.$$ If we take $g(m)=E(m)$, then we have 
$$ \frac{1}{|S|}\sum_{x\in S} |\widehat{E}(x)|^2 \lesssim |E|^{\frac{d+2}{d}}$$
Since $|S|\sim q^{d-1},$    (\ref{aimcor}) holds and the proof is complete.
\end{proof}

{\bf Acknowledgement :} The authors would like to thank the referee for his/her valuable comments for developing the final version of this paper.

\end{document}